\documentclass[reqno]{amsart}
\usepackage{graphicx} 
\usepackage[total={6.5in, 9in}]{geometry}

\usepackage{amsfonts, amsaddr}
\usepackage{amssymb}
\usepackage{color}
\usepackage{amsmath}
\usepackage{amsthm}
\usepackage{multicol}
\usepackage{mathtools}
\usepackage{tikz,tikz-cd, tcolorbox}
\usepackage[colorlinks=true,linkcolor=blue,citecolor=blue]{hyperref}
\usepackage{setspace, mathrsfs}
\usepackage{bussproofs}
\theoremstyle{plain}
\newtheorem{theorem}{Theorem}[section]
\newtheorem{lemma}[theorem]{Lemma}
\newtheorem{proposition}[theorem]{Proposition}
\newtheorem{corollary}[theorem]{Corollary}

\theoremstyle{definition}
\newtheorem{definition}[theorem]{Definition}
\newtheorem{remark}[theorem]{Remark}
\newtheorem{example}[theorem]{Example}

\begin{document}
\title[Monadic and cylindric expansions of bounded implication algebras]{Monadic and cylindric expansions of bounded implication algebras}
\author{Joseph McDonald }
\address{Department of Theoretical Computer Science\\Institute of Computer Science, Czech Academy of Sciences}

\email{mcdonald@cs.cas.cz}
\date{May 2026}
\maketitle

\begin{abstract}
Implication algebras were introduced by Abbott as algebraic models of the operation of Boolean implication in the classical propositional calculus. In this work, we study additional operators and constants on bounded implication algebras by introducing monadic and cylindric implication algebras. It is demonstrated that the category $\mathbf{MIA}$ of monadic implication algebras is isomorphic to the category $\mathbf{MBA}$ of monadic Boolean algebras and moreover, that the category $\mathbf{CIA}$ of $I$-dimensional cylindric implication algebras is isomorphic to the category $\mathbf{CBA}$ of $I$-dimensional cylindric Boolean algebras. As an application of the obtained categorical isomorphisms, we provide spectral duality results for $I$-dimensional cylindric implication algebras along the lines of Bezhanishvili and Holliday's spectral duality for Boolean algebras combined with McDonald's extension of their duality to monadic and $I$-dimensional cylindric Boolean algebras. 
\par
\vspace{.2cm}
\end{abstract}



\maketitle

\section{Introduction}
Implication algebras were introduced by Abbott \cite{abbott672} as algebraic models of the operation of Boolean implication in the classical propositional calculus. An \emph{implication algebra} is a binar $\langle A;\cdot\rangle$ satisfying the contraction, quasi-commutativity, and exchange equations. Abbott demonstrated that every implication algebra can be converted into a semi-Boolean algebra (i.e., a join-semilattice with a greatest element in which every principal filter forms a Boolean algebra) and conversely, that every semi-Boolean algebra gives rise to an implication algebra under the semi-Boolean algebra polynomial $\pi(x,y):=(x\vee y)'_y$ where $(x\vee y)'_y$ is the Boolean complement of $x\vee y$ in the principal filter generated by $y$. Abbott additionally stated (without giving an explicit proof) that by adjoining a certain constant to an implication algebra (resulting in a bounded implication algebra), which determines the least element in the induced join-semilattice, then one obtains a Boolean algebra, and moreover, that every Boolean algebra $B$ can be converted into a bounded implication algebra $\langle B;\pi_c,0\rangle$ under $\pi_c(x,y):=x'\vee y$. Section 2 of this paper is concerned with giving explicit constructions of these latter claims and then extending these results to that of an isomorphism between the category $\mathbf{IA_{0,1}}$ of bounded implication algebras and the category $\mathbf{BA}$ of Boolean algebras.    

We then investigate bounded implication algebras within the setting of monadic and cylindric Boolean algebras. Monadic Boolean algebras were introduced by Halmos \cite{halmos} as algebraic models of the classical single-variable predicate calculus. A \emph{monadic Boolean algebra} consists of a Boolean algebra $B$ equipped with a closure operator $\exists\colon B\to B$, known as a \emph{quantifier}, whose closed elements form a Boolean subalgebra. Cylindric Boolean algebras were introduced by Henkin, Monk, and Tarski \cite{tarski1, tarski2} as algebraic models of the classical $n$-variable predicate calculus with identity. An \emph{I-dimensional cylindric Boolean algebra} consists of a Boolean algebra $B$ equipped with a family $(\exists_i)_{i\in I}\colon B\to B$ of pairwise commuting quantifiers and a family $(\delta_{i,k})_{i,k\in I}$ of constants, known as the \emph{diagonal elements}, which are subject to certain conditions that algebraically model identity in a first-order language. Closely related to the monadic and cylindric Boolean algebras are the polyadic Boolean algebras introduced by Halmos \cite{halmos} as algebraic models of the classical $n$-variable predicate calculus without identity.  

In Section 3, we introduce monadic implication algebras. After establishing some of their basic algebraic properties, we demonstrate that the category $\mathbf{MIA}$ of monadic implication algebras is isomorphic to the category $\mathbf{MBA}$ of monadic Boolean algebras. Monadic implication algebras have been considered elsewhere (such as in \cite{abad}), however they study monadic implication algebras from the perspective of universal quantifiers (i.e., interior operators whose open elements form  Boolean subalgebras), whereas we study monadic implication algebras from the perspective of existential quantifiers. More importantly, whereas the monadic implication algebras considered here are bounded, the monadic implication algebras considered in \cite{abad} are not assumed to be bounded. Finally, we introduce $I$-dimensional cylindric implication algebras. It is demonstrated that the category $\mathbf{CIA}$ of $I$-dimensional cylindric implication algebras is isomorphic to the category $\mathbf{CBA}$ of $I$-dimensional cylindric Boolean algebras.

In Section 4, as an application of the obtained categorical isomorphisms, we provide a duality result for $I$-dimensional cylindric implication algebras along the lines of Bezhanishvili and Holliday's spectral duality for Boolean algebras \cite{bez} combined with McDonald's extension of their duality to monadic Boolean algebras and $I$-dimensional cylindric Boolean algebras \cite{mcdonald2}.         

\section{Bounded implication algebras}
In this section, we describe the implication algebras introduced in \cite{abbott672} and then give a explicit construction of, and extend, \cite[Theorem 19]{abbott672} to an isomorphism between $\mathbf{IA_{0,1}}$ and $\mathbf{BA}$. 
\subsection{Implication algebras and semi-Boolean algebras} In this subsection, we discuss some basic details of implication algebras and describe the construction given in \cite{abbott672} which associates to every implication algebra, a semi-Boolean algebra. 
\begin{definition}\label{bia}
    An \emph{implication algebra} is a binar $\langle A;\cdot\rangle$ satisfying:
    \begin{enumerate}
        \item $(x\cdot y)\cdot x=x$ \hfill (contraction)
        \item $(x\cdot y)\cdot y=(y\cdot x)\cdot x$ \hfill (quasi-commutativity)
        \item $x\cdot (y\cdot z)=y\cdot(x\cdot z)$ \hfill (exchange)
    \end{enumerate}
\end{definition}
The following are immediate consequences of the axioms of implication algebras. 
\begin{proposition}[\cite{abbott672}]\label{prop1}
    Any implication algebra satisfies: 
    \begin{enumerate}
        \item $x\cdot (x\cdot y)=x\cdot y$
        \item $x\cdot x=(x\cdot y)\cdot(x\cdot y)$
    \end{enumerate}
    \begin{proof}
        Definition \ref{bia}(1) yields $x\cdot(x\cdot y)=((x\cdot y)\cdot x)\cdot(x\cdot y)=x\cdot y$ which implies $x\cdot x=((x\cdot y)\cdot x)\cdot x=(x\cdot(x\cdot y))\cdot(x\cdot y)=(x\cdot y)\cdot(x\cdot y)$. This completes the proof.     
    \end{proof}
\end{proposition}

\begin{proposition}[\cite{abbott672}]\label{constant1}
    Any implication algebra $A$ admits of a constant $c_1\in A$ satisfying $x\cdot x=c_1$, $x\cdot c_1=c_1$, and $c_1\cdot x=x$. 
\end{proposition}
\begin{proof}
Proposition \ref{prop1} along with Definition \ref{bia}(2) yields the following equality: \[x\cdot x=(x\cdot y)\cdot(x\cdot y)=((x\cdot y)\cdot y)\cdot((x\cdot y)\cdot y)=((y\cdot x)\cdot x)\cdot((y\cdot x)\cdot x)=(y\cdot x)\cdot(y\cdot x)=x\cdot x.\] This implies the existence of a constant $c_1=x\cdot x$ that is independent of $x$ so that $c_1$ is well-defined. Moreover, we have $c_1\cdot x=(x\cdot x)\cdot x=x$ by Definition \ref{bia}(1) and $x\cdot c_1=x\cdot(x\cdot x)=x\cdot x=c_1$ by Proposition \ref{prop1}(1).  
\end{proof}
 
\begin{lemma}[\cite{abbott672}]\label{bia to poset}
    Every implication algebra $A$ determines a poset $\langle A;\preceq\rangle$ under $x\preceq y$ iff $x\cdot y=c_1$. Moreover, the constant $c_1\in A$ is the greatest element in this poset. In addition, $\langle A;\oplus\rangle$ is a join-semilattice under $x\oplus y:=(x\cdot y)\cdot y$.    
\end{lemma}

    A \emph{semi-Boolean algebra} is a bounded join-semilattice $\langle B;\vee,1\rangle$ such that ${\uparrow}x=\{y\in B:x\leq y\}$ is a Boolean algebra for each $x\in B$. The following associates to each implication algebra, a semi-Boolean algebra.    

\begin{theorem}[\cite{abbott672}]\label{bia to sba}
    Let $A$ be an implication algebra and let $\langle A;\oplus,c_1\rangle$ be the bounded join-semilattice obtained from $A$ via Lemma \ref{bia to poset}. Then $\langle {\uparrow}z;\oplus,\odot,-,z,c_1\rangle$ is a Boolean algebra for any $z\in A$ under $x\odot y:=(x\cdot(y\cdot z))\cdot z$ and $-x:=x\cdot z$.   
\end{theorem}
\begin{theorem}[\cite{abbott672}]\label{sba to bia}
    Let $A$ be a semi-Boolean algebra and let $z\in A$. Then $\langle{\uparrow}z;\pi_c^*\rangle$ is an implication algebra under $\pi_c^*(x,y):=(x\vee y)'_z$ where $(x\vee y)'_z$ is the Boolean complement of $x\vee y$ in the Boolean algebra ${\uparrow}z$. 
\end{theorem}
\subsection{Isomorphism between $\mathbf{IA_{0,1}}$ and $\mathbf{BA}$} We proceed by extending Theorems \ref{bia to sba} and \ref{sba to bia} to an isomorphism between the category $\mathbf{IA_{0,1}}$ of bounded implication algebras and the category $\mathbf{BA}$ of Boolean algebras. 
\begin{definition}\label{definition 2.7}
    A \emph{bounded implication algebra} is an implication algebra $A$ equipped with a constant element $c_0\in A$ satisfying $c_0\cdot x=1$ for all $x\in A$. 
\end{definition}

    Note that for any bounded implication algebra $A$, the constant $c_0\in A$ is the least element in the poset $\langle A;\preceq\rangle$. The following is due to Abbott \cite{abbott672} but was stated without proof. Hence, we fill in the details for the sake of clarity. 
\begin{theorem}\label{ba to bia}
    If $B$ is a Boolean algebra, $\langle B;\pi_c\rangle$ is a bounded implication algebra under $\pi_c(x,y):=x'\vee y$. 
\end{theorem}
\begin{proof}
    We first verify that $\pi_c$ satisfies the contraction equation by checking that $\pi_c(\pi_c(x,y),x)=x$. The definition of $\pi_c$, De Morgan's identities, the absorption identities, and the fact that $'$ is an involution yields: \[\pi_c(\pi_c(x,y),x)=\pi_c(x'\vee y,x)=(x'\vee y)'\vee x=(x''\wedge y')\vee x=(x\wedge y')\vee x=x.\] We now verify the quasi-commutativity equation by checking that  $\pi_c(\pi_c(x,y),y)=\pi_c(\pi_c(y,x),x)$:
    \begin{align*}
        \pi_c(\pi_c(x,y),y)&=\pi_c((x'\vee y),y)=(x'\vee y)'\vee y=(x''\wedge y')\vee y=(x\wedge y')\vee y\\&=(x\vee y)\wedge(y'\vee y)=(x\vee y)\wedge 1=x\vee y.
    \end{align*}
    Since $x\vee y=y\vee x$, reversing the above calculation applied to $y\vee x$ yields $y\vee x=\pi_c(\pi_c(y,x),x)$ and hence we obtain $\pi_c(\pi_c(x,y),y)=\pi_c(\pi_c(y,x),x)$. Finally, we now verify that $\pi_c$ satisfies the exchange equation by checking that $\pi_c(x,\pi_c(y,z))=\pi_c(y,\pi_c(x,z))$. The calculation proceeds as follows:
    \begin{align*}
        \pi_c(x,\pi_c(y,z))&=\pi_c(x,y'\vee z)=x'\vee(y'\vee z)=(x'\vee y')\vee z=(y'\vee x')\vee z
        \\&=y'\vee(x'\vee z)=y'\vee(\pi_c(x,z))=\pi_c(y,\pi_c(x,z)).
    \end{align*}
    This completes the proof that $\langle B;\pi_c\rangle$ is an implication algebra. To see that $\langle B;\pi_c\rangle$ is bounded, observe that $\pi_c(0,x)=0'\vee x=1\vee x=1$ for all $x\in B$.  
\end{proof}
\begin{theorem}\label{bia to ba}
    Every bounded implication algebra can be converted into a Boolean algebra. 
\end{theorem}
\begin{proof}
    The proof follows from Theorem \ref{bia to sba} by setting $z=c_0$ so that ${\uparrow}z=A$.  
\end{proof}
\begin{theorem}\label{theorem 2.10}
    Let $\mathfrak{B}=\langle B;\wedge,\vee,',0,1\rangle$ be a Boolean algebra, let $\mathfrak{B}^+=\langle B;\pi_c,0\rangle$ be the bounded implication algebra obtained from $\mathfrak{B}$ via Theorem \ref{ba to bia}, and let $(\mathfrak{B}^{+})_+=\langle B;\odot,\oplus,-,0,1\rangle$ be the Boolean algebra obtained from $\mathfrak{B}^+$ via Theorem \ref{bia to ba}. Then $\mathfrak{B}=(\mathfrak{B}^+)_+$. 
\end{theorem}
\begin{proof}
    Note that the carrier set of $\mathfrak{B}$ is identical to the carrier set of $(\mathfrak{B}^+)_+$ and $x\leq y$ iff $\pi_c(x,y)=1$ iff $x\preceq y$. Moreover, $x'=x'\vee0=\pi_c(x,0)=-x$ as well as:
    \begin{align*}
        x\wedge y&=(x'\vee y')'=(x'\vee(y'\vee 0))'=(x'\vee(y'\vee 0))'\vee 0=\pi_c((x'\vee(y'\vee0)),0)
        \\&=\pi_c(\pi_c(x,y'\vee 0),0)=\pi_c(\pi_c(x,\pi_c(y,0)),0)=x\odot y
    \end{align*}
    \begin{align*}
        x\vee y&=(x\vee y)\wedge 1=(x\vee y)\wedge(y\vee y')=(x''\vee y)\wedge(y\vee y')=(x''\wedge y')\vee y
        \\&=(x'\vee y)'\vee y=\pi_c(x'\vee y,y)=\pi_c(\pi_c(x,y),y)=x\oplus y.
    \end{align*}
    This completes the proof.
\end{proof} 
Th following will be exploited in the proof of Theorem \ref{theorem 2.12}. 
\begin{lemma}[\cite{abbott672}]\label{lemma 1}
    In any implication algebra, if $z\cdot y=c_1$, then $x\cdot y=((x\cdot z)\cdot y)\cdot y$.  
\end{lemma}
\begin{theorem}\label{theorem 2.12}
    Let $\mathfrak{A}=\langle A;\cdot,c_0\rangle$ be a bounded implication algebra, let $\mathfrak{A}_+=\langle A;\odot,\oplus,-,c_0,c_1\rangle$ be the Boolean algebra obtained from $\mathfrak{A}$ via Theorem \ref{bia to ba}, and let $(\mathfrak{A}_+)^+=\langle A;\pi_c,c_0\rangle$ be the bounded implication algebra obtained from $\mathfrak{A}_+$ via Theorem \ref{ba to bia}. Then $\mathfrak{A}=(\mathfrak{A}_+)^+$. 
\end{theorem}
\begin{proof}
    Observe that since $c_0\cdot y=c_1$, by Lemma \ref{lemma 1} we have: \[x\cdot y=((x\cdot c_0)\cdot y)\cdot y=(x\cdot c_0)\oplus y=x^*\oplus y=\pi_c(x,y)\] which completes the proof. 
\end{proof}
\begin{definition}
    Let $A$ and $B$ be bounded implication algebras. A function $h\colon A\to B$ is a \emph{homomorphism} provided $h(x\cdot y)=h(x)\cdot h(y)$ and $h(c_0)=c_0$. 
\end{definition}
Notice that the above notion of homomorphism implies $h(c_1)=c_1$ since $h(c_1)=h(x\cdot x)=h(x)\cdot h(x)=c_1$. 
\begin{definition}
    By $\mathbf{IA_{0,1}}$ we denote the category of bounded implication algebras and bounded implication algebra homomorphisms and by $\mathbf{BA}$ we denote the category of Boolean algebras and Boolean algebra homomorphisms. 
\end{definition}
\begin{theorem}\label{theorem 2.15}
    $\mathbf{IA_{0,1}}$ is isomorphic to $\mathbf{BA}$. 
\end{theorem}
\begin{proof}
    By Theorem \ref{theorem 2.10} and Theorem \ref{theorem 2.12}, it suffices to exhibit a bijective correspondence between homomorphisms of bounded implication algebras and homomorphisms of Boolean algebras. We have: \[h(\pi_c(x,y))=h(x'\vee y)=h(x')\vee h(y)=h(x)'\vee h(y)=\pi_c(h(x)',h(y))\] \[h(x\oplus y)=h((x\cdot y)\cdot y)=h(x\cdot y)\cdot h(y)=(h(x)\cdot h(y))\cdot h(y)=h(x)\oplus h(y).\] The remaining cases run analogously.   
\end{proof}
\section{Monadic and cylindric implication algebras}
In this section, we extend the results collected in Section 2 to an isomorphism between $\mathbf{MIA}$ and $\mathbf{MBA}$ and then between $\mathbf{CIA}$ and $\mathbf{CBA}$. 
\subsection{Monadic and cylindric Boolean algebras} In this subsection, we describe some of the basic aspects of monadic and cylindric Boolean algebras. For more details pertaining to the former, consult \cite{halmos} and for the latter, consult \cite{tarski1, tarski2}.   

\begin{definition}\label{monadic ortholattice}
A \emph{monadic Boolean algebra} is an algebra $\langle A;\wedge,\vee,',0,1,\exists\rangle$ such that: 
\begin{enumerate}
    \item $\langle A;\wedge,\vee,',0,1\rangle$ is a Boolean algebra;
    \item $\exists\colon A\to A$ is a unary operator, known as a \emph{quantifier}, satisfying:
        \begin{enumerate}
        \item $\exists 0=0$ and $x\leq \exists x$;
        \item $\exists(x\wedge\exists y)=\exists x\wedge\exists y$. 
    \end{enumerate}
\end{enumerate}
\end{definition}
The motivating example of a monadic Boolean algebra arises by considering the function space from a non-empty set into a complete Boolean algebra. 
\begin{example}
    Let $X$ be a non-empty set and let $B$ be a complete Boolean algebra. Then the algebra $\langle B^X;\cdot,+,-,c_0,c_1\rangle$ where  $B^X$ is the function space from $X$ to $B$ and 
    \begin{enumerate} 
        \item $(f\cdot g)(x)=f(x)\wedge g(x)$; 
        \item $(f+g)(x)=f(x)\vee g(x)$; 
        \item $-f(x)=f(x)'$; 
        \item $c_0(x)=0$, $c_1(x)=1$, 
    \end{enumerate}
    is a Boolean algebra, known as the \emph{full functional Boolean algebra} determined by $X$ and $B$. Defining: 
    \[\exists\colon B^X\to B^X;\hspace{.4cm}(\exists f)(x)=\bigvee\text{ran}(f)=\bigvee\{f(x):x\in X\},\] induces a quantifier on $B^X$ making $\langle B^X;\cdot,+,-,c_0,c_1,\exists\rangle$ a monadic Boolean algebra, known as the \emph{full functional monadic Boolean algebra} determined by $X$ and $B$.   
\end{example}

\begin{definition}\label{quantum cylindric algebra}
    An \emph{I-dimensional cylindric Boolean algebra} is an algebra: \[\langle A;\wedge,\vee,',0,1,(\exists_i)_{i\in I},(\delta_{i,k})_{i,k\in I}\rangle\] satisfying the following conditions: 
    \begin{enumerate}
    \item $\langle A;\wedge,\vee,',0,1,\exists_i\rangle$ is a monadic Boolean algebra for each $i\in I$; 
        \item $\exists_i\exists_kx=\exists_k\exists_ix$ for all $i,k\in I$; 
        \item $(\delta_{i,k})_{i,k\in I}$ is a family of constants, known as the \emph{diagonal elements}, satisfying: 
            \begin{enumerate}
        \item $\delta_{i,k}=\delta_{k,i}$ and $\delta_{i,i}=1$;
        \item $i,l\not=k\Rightarrow\exists_{k}(\delta_{i,k}\wedge \delta_{k,l})=\delta_{i,l}$; 
        \item $i\not=k\Rightarrow\exists_i(d_{i,k}\wedge x)\wedge\exists_i(d_{i,k}\wedge x')=0$. 
    \end{enumerate}
    \end{enumerate}
    We call the reduct $\langle A;\wedge,\vee,^{\perp},0,1,(\exists_{i})_{i\in I}\rangle$ the \emph{I-dimensional diagonal-free cylindric Boolean algebra reduct} of $A$ provided conditions 1-2 are satisfied.   
\end{definition}
\begin{example}
     $I$-dimensional cylindric Boolean algebras arise from taking the Boolean algebra induced by the function space $X^I$ where $X$ is a non-empty set. Each $f\in X^I$ is an $I$-indexed family and for any subset $F\subseteq X^I$, its cylindrification $\exists_iF$ is the set of all such indexed families $g$ that agree with some $f\in F$ expect possibly at the coordinate $i$. The diagonal $\delta_{i,k}$ is then given by $d_{i,k}=\{f:f(i)=f(k)\}$.  
\end{example}

\subsection{Monadic implication algebras} We now introduce monadic implication algebras and establish some of their basic properties. 

\begin{definition}\label{mbia}
    A \emph{monadic implication algebra} is a bounded implication algebra $A$ equipped with an additional operator $\nabla\colon A\to A$ satisfying the following conditions: 
    \begin{enumerate} 
        \item $\nabla c_0=c_0$ and $x\cdot\nabla x=c_1$;
        \item $\nabla((x\cdot(\nabla y\cdot c_0))\cdot c_0)=((\nabla x\cdot(\nabla y\cdot c_0))\cdot c_0)$. 
    \end{enumerate}
\end{definition}
The following  collects some useful facts about monadic implication algebras. 
\begin{proposition}\label{idempotent}
    In any monadic implication algebra, $\nabla$ preserves the constant $c_1$ and is idempotent.  
\end{proposition}
\begin{proof}
    Part 2 of Definition \ref{mbia}(1) yields $c_1\cdot \nabla c_1=c_1$ and Proposition \ref{constant1} gives $\nabla c_1\cdot c_1=c_1$ and hence $\nabla c_1=c_1$. To see that $\nabla$ is idempotent, observe that by Definition \ref{definition 2.7}, Proposition \ref{constant1}, and Definition \ref{mbia}(2): 
\begin{equation}
    \nabla x=c_1\cdot \nabla x=(c_0\cdot \nabla x)\cdot \nabla x=(\nabla x\cdot c_0)\cdot c_0=(c_1\cdot(\nabla x\cdot c_0))\cdot c_0
\end{equation}
    Since $\nabla c_1=c_1$, Definition \ref{mbia}(2) along with Equation (1) yields: \[\nabla\nabla x=\nabla((c_1\cdot(\nabla x\cdot c_0))\cdot c_0)=(\nabla c_1\cdot(\nabla x\cdot c_0))\cdot c_0=(c_1\cdot(\nabla x\cdot c_0))\cdot c_0=\nabla x.\] This completes the proof. 
\end{proof}
If $A$ is a monadic implication algebra, let $\text{ran}(\nabla):=\{\nabla x:x\in A\}$. 
\begin{proposition}\label{prop 3.7}
    In any monadic implication algebra $A$, we have: \[\text{ran}(\nabla)=\{x\in A:\nabla x=x\}.\] 
\end{proposition}
\begin{proof}
    Suppose that $x\in\text{ran}(\nabla)$ so that $x=\nabla y$ for some $y\in A$. Then since $\nabla$ is idempotent by Proposition \ref{idempotent}, we obtain $\nabla x=\nabla\nabla y=\nabla y=x$, as desired. The $\{x\in A:\nabla x=x\}\subseteq\text{ran}(\nabla)$ inclusion is trivial.   
\end{proof}
\begin{proposition}\label{prop 3.8}
    Any monadic implication algebra satisfies: 
    \begin{enumerate}
        \item if $x\cdot\nabla y=c_1$, then $\nabla x\cdot\nabla y=c_1$; 
        \item if $x\cdot y=c_1$, then $\nabla x\cdot\nabla y=c_1$. 
    \end{enumerate}
\end{proposition}
\begin{proof}
    For condition 1, assume that $x\cdot\nabla y=c_1$ so that by Theorem \ref{bia to ba} we have $(x\cdot(\nabla y\cdot c_0))\cdot c_0=x$. Then by Definition \ref{mbia}(2) together with the fact that $\nabla$ is an involution yields: \[\nabla x=\nabla(x\cdot(\nabla y\cdot c_0))\cdot c_0=(\nabla x\cdot(\nabla\nabla y\cdot c_0))\cdot c_0=(\nabla x\cdot(\nabla y\cdot c_0))\cdot c_0\] which by Theorem \ref{bia to ba} implies $\nabla x\cdot\nabla y=c_1$. For condition 2, assume that $x\cdot y=c_1$ and observe that we have $y\cdot\nabla y=c_1$ by Definition \ref{mbia}(1). Then $x\cdot\nabla y=c_1$ and hence by the first part of this this proof, we have $\nabla x\cdot\nabla y=c_1$.    
\end{proof}
\begin{proposition}\label{prop 3.9}
    Any monadic implication algebra satisfies $\nabla(\nabla x\cdot c_0)=\nabla x\cdot c_0$. 
\end{proposition}
\begin{proof}
Since $(\nabla x\cdot c_0)\cdot\nabla(\nabla x\cdot c_0)=c_1$ by Definition \ref{mbia}(1), it remains to show $\nabla(\nabla x\cdot c_0)\cdot(\nabla x\cdot c_0)=1$. Two applications of Proposition \ref{constant1} combined with Definition Definition \ref{mbia}(1) yields the following:   
\begin{equation}
    ((\nabla x\cdot c_0)\cdot(\nabla x\cdot c_0))\cdot c_0=c_1\cdot c_0=c_0=\nabla(((\nabla x\cdot c_0)\cdot(\nabla x\cdot c_0))\cdot c_0)=(\nabla(\nabla x\cdot c_0)\cdot(\nabla x\cdot c_0))\cdot c_0.
\end{equation}
Equation (2) then implies that $\nabla(\nabla x\cdot c_0)\cdot(\nabla x\cdot c_0))=c_1$, as desired. 
\end{proof}

\begin{proposition}\label{prop 3.10}
   If $A$ is a monadic implication algebra, $x,y\in\text{ran}(\nabla)$ implies $(x\cdot(y\cdot c_0))\cdot c_0,\hspace{.1cm} x\cdot c_0\in\text{ran}(\nabla)$.  
\end{proposition}
\begin{proof}
    Suppose that $x,y\in\text{ran}(\nabla)$ so that $x=\nabla x$ and $y=\nabla y$ by Proposition \ref{prop 3.7}. Definition \ref{mbia}(2) gives:
    \[(x\cdot(y\cdot c_0))\cdot c_0=(\nabla x\cdot(\nabla y\cdot c_0))\cdot c_0=\nabla((x\cdot(\nabla y\cdot c_0))\cdot c_0)\] with $\nabla((x\cdot(\nabla y\cdot c_0))\cdot c_0)\in\text{ran}(\nabla)$. Moreover, Proposition \ref{prop 3.9} yields $x\cdot c_0=\nabla x\cdot c_0=\nabla(\nabla x\cdot c_0)\in\text{ran}(\nabla)$.  
\end{proof}
\begin{proposition}\label{preservation of joins}
    Any monadic implication algebra satisfies $\nabla((x\cdot y)\cdot y)=((\nabla x\cdot\nabla y)\cdot\nabla y)$.  
\end{proposition}
\begin{proof}
    By Lemma \ref{bia to poset}, we have $x\cdot((x\cdot y)\cdot y)=y\cdot((x\cdot y)\cdot y)=c_1$ and hence by Proposition \ref{prop 3.8}(2), we have: \[\nabla x\cdot\nabla((x\cdot y)\cdot y)=\nabla y\cdot\nabla((x\cdot y)\cdot y)=c_1\] so $((\nabla x\cdot\nabla y)\cdot\nabla y)\cdot\nabla((x\cdot y)\cdot y)=c_1$. Now observe that $\nabla x,\nabla y\in\text{ran}(A)$ so by Theorem \ref{bia to ba}, Proposition \ref{prop 3.10} along with the De Morgan identities, $(\nabla x\cdot\nabla y)\cdot\nabla y\in\text{ran}(\nabla)$ and hence by Proposition \ref{prop 3.7} we have:
    \begin{equation}
        (\nabla x\cdot\nabla y)\cdot\nabla y=\nabla((\nabla x\cdot\nabla y)\cdot\nabla y).
    \end{equation}
     Then since $x\cdot((\nabla x\cdot\nabla y)\cdot\nabla y)=y\cdot((\nabla x\cdot\nabla y)\cdot\nabla y)=c_1$ by Definition \ref{mbia}(1), we have $((x\cdot y)\cdot y)\cdot((\nabla x\cdot\nabla y)\cdot\nabla y)=c_1$ which by Proposition \ref{prop 3.8}(2) yields $\nabla((x\cdot y)\cdot y)\cdot\nabla((\nabla x\cdot\nabla y)\cdot\nabla y)=c_1$. Then by Equation (3) we obtain $\nabla((x\cdot y)\cdot y)\cdot(\nabla x\cdot\nabla y)\cdot\nabla y=c_1$ and hence conclude $\nabla((x\cdot y)\cdot y)=((\nabla x\cdot\nabla y)\cdot\nabla y)$.  
\end{proof}

\subsection{Isomorphism between $\mathbf{MIA}$ and $\mathbf{MBA}$} In this subsection, we demonstrate that the category $\mathbf{MIA}$ of monadic implication algebras is isomorphic to the category $\mathbf{MBA}$ of monadic Boolean algebras. 

\begin{theorem}\label{mba to mia}
    Every monadic Boolean algebra can be converted into a monadic implication algebra. 
\end{theorem}
\begin{proof}
    By Theorem \ref{ba to bia}, every Boolean algebra $B$ gives rise to a bounded implication algebra $\langle B;\pi_c,0\rangle$ and hence it suffices to demonstrate that the quantifier $\exists\colon B\to B$ satisfies conditions 1 and 2 of Definition \ref{mbia}. For condition 1, observe that $\exists 0=0$ is trivial and that since $\exists$ is increasing and $\vee$ is monotone, we have $\pi_c(x,\exists x)=x'\vee\exists x=1$. For condition 2, if suffices to demonstrate the following: 
    \begin{equation}
\exists\pi_c(\pi_c(x,\pi_c(y,0)),0)=\pi_c(\pi_c(\exists x,\pi_c(\exists y,0)),0).
    \end{equation}
    The first part of calculation proceeds by a repeated application of the definition of $\pi_c$ along with the fact that $\pi_c(x,0)=x'$ for all $x\in B$:
    \begin{equation}
        \exists(\pi_c(\pi_c(x,\pi_c(\exists y,0)),0)=\exists(\pi_c(\pi_c(x,(\exists y)')),0)=\exists(\pi_c(x,(\exists y)')')=\exists(x'\vee(\exists y)')'
    \end{equation}
    De Morgan's identities, the fact that $'$ is an involution, Definition \ref{monadic ortholattice}(2(b)) gives:
    \begin{equation}
        \exists(x'\vee(\exists y)')'=\exists(x''\wedge(\exists y)'')=\exists(x\wedge\exists y)=\exists x\wedge \exists y
    \end{equation}
Applying De Morgan's identities and the fact that $0$ is a unit then yields: 
\begin{equation}
     \exists x\wedge \exists y=((\exists x)'\vee (\exists y)')'=((\exists x)'\vee((\exists y)'\vee 0))'=((\exists x)'\vee((\exists y)'\vee 0))'\vee 0
\end{equation}
Applying the definition of $\pi_c$ then yields the following: 
\begin{equation}
    ((\exists x)'\vee((\exists y)'\vee 0))'\vee 0=\pi_c(\pi_c(\exists x,\pi_c(\exists y,0)),0)
\end{equation}
Equations (5)-(8) then give Equation (4), as desired.    
\end{proof}
\begin{theorem}\label{theorem 3.13}
    Every monadic implication algebra can be converted into a monadic Boolean algebra. 
\end{theorem}
\begin{proof}
 We have $\nabla c_0=c_0$ as well as $x\cdot\nabla x=c_1$ and hence $x\preceq\nabla x$ and thus Definition \ref{monadic ortholattice}(2(a)) is satisfied. To see that condition 2(b) is satisfied, note that: \[\nabla(x\odot\nabla y)=\nabla((x\cdot(\nabla y\cdot c_0))\cdot c_0)=(\nabla x\cdot(\nabla y\cdot c_0))\cdot c_0=\nabla x\odot\nabla y,\] where the second equality follows by a direct application of Definition \ref{mbia}(2). This completes the proof.     
\end{proof}
\begin{corollary}
    Let $\mathfrak{A}=\langle A;\cdot,c_0,\nabla\rangle$ be a monadic implication algebra and let $\mathfrak{A}_+=\langle A;\odot,\oplus,-,c_0,c_1,\nabla\rangle$ be the monadic Boolean algebra induced by $\mathfrak{A}$ via Theorem \ref{theorem 3.13}. Then $\text{ran}(\nabla)$ forms a Boolean subalgebra of $\mathfrak{A}_+$. 
\end{corollary}
\begin{proof}
    By Proposition \ref{prop 3.10}, we have $(x\cdot (y\cdot c_0))\cdot c_0\in\text{ran}(\nabla)$ and $x\cdot c_0\in\text{ran}(\nabla)$ whenever $x,y\in\text{ran}(\nabla)$. Moreover, since $c_1=\nabla c_1$, it follows that $\text{ran}(\nabla)$ is closed under meets, Boolean complements, and the top universal bound. Applying De Morgan's identities to the induced Boolean algebra $\mathfrak{A}_+$ from $\mathfrak{A}$ then implies that $\text{ran}(\nabla)$ is closed under the relevant Boolean algebra operations.  
\end{proof}
The results collected thus far yield the following theorems. 
\begin{theorem}\label{theorem 3.15}
    Let $\mathfrak{B}=\langle A;\wedge,\vee,',0,1,\exists\rangle$ be a monadic Boolean algebra, let $\mathfrak{B}^+=\langle A;\pi_c,0,\exists\rangle$ be the monadic implication algebra obtained from $\mathfrak{B}$ via Theorem \ref{mba to mia}, and let $(\mathfrak{B}^{+})_+=\langle A;\odot,\oplus,-,0,1,\exists\rangle$ be the monadic Boolean algebra obtained from $\mathfrak{B}^+$ via Theorem \ref{theorem 3.13}. Then $\mathfrak{B}=(\mathfrak{B}^+)_+$. 
\end{theorem}
\begin{theorem}\label{theorem 3.16}
    Let $\mathfrak{A}=\langle A;\cdot,c_0,\nabla\rangle$ be a monadic implication algebra, let $\mathfrak{A}_+=\langle A;\odot,\oplus,-,c_0,c_1,\nabla\rangle$ be the monadic Boolean algebra obtained from $\mathfrak{A}$ via Theorem \ref{theorem 3.13}, and let $(\mathfrak{A}_+)^+=\langle A;\pi_c,c_0,\nabla\rangle$ be the monadic implication algebra obtained from $\mathfrak{A}_+$ via Theorem \ref{mba to mia}. Then $\mathfrak{A}=(\mathfrak{A}_+)^+$. 
\end{theorem}
\begin{definition}
    If $A$ and $B$ are monadic implication algebras, then a function $h\colon A\to B$ is a \emph{homomorphism} provided $h$ is a homomorphism from the bounded implication algebra reduct $\langle A;\cdot,c_0\rangle$ of $A$ to the bounded implication algebra reduct $\langle B;\cdot,c_0\rangle$ of $B$ and satisfies $h(\nabla x)=\nabla h(x)$. 
\end{definition}

\begin{theorem}\label{theorem 3.18}
    $\mathbf{MIA}$ is isomorphic to $\mathbf{MBA}$. 
\end{theorem}
\begin{proof}
    Apply Theorem \ref{theorem 2.15}, Theorem \ref{theorem 3.15}, and Theorem \ref{theorem 3.16}. 
\end{proof}
\subsection{Isomorphism between $\mathbf{CIA}$ and $\mathbf{CBA}$}
We now introduce the cylindric expansions of the monadic implication algebras and demonstrate that $\mathbf{CIA}$ is isomorphic to $\mathbf{CBA}$.  
\begin{definition}\label{cia}
    An \emph{I-dimensional cylindric implication algebra} is an algebraic structure with the following signature $\langle A;\cdot,c_0,(\nabla_i)_{i\in I}, (d_{i,k})_{i,k\in I}\rangle$ such that the following conditions are satisfied: 
    \begin{enumerate}
        \item $\langle A;\cdot,c_0,\nabla_i\rangle$ is a monadic implication algebra for each $i\in I$; 
        \item $(\nabla_i)_{i\in I}$ is a family of pairwise commuting operators, i.e., $\nabla _i\nabla_kx=\nabla_k\nabla_ix$; 
        \item $(d_{i,k})_{i,k\in I}$ is a family of constants satisfying: 
        \begin{enumerate}
            \item $d_{i,k}=d_{k,i}$ and $d_{i,i}=c_1$; 
            \item $i,l\not=k\Rightarrow\nabla_{k}\bigl[(d_{i,k}\cdot (d_{k,l}\cdot c_0))\cdot c_0\bigl]=d_{i,l}$; 
            \item $i\not=k\Rightarrow\biggl[\nabla_i\biggl((d_{i,k}\cdot(x\cdot c_0))\cdot c_0\biggl)\cdot\biggl(\biggl(\nabla_i\biggl(\bigl[d_{i,k}\cdot((x\cdot c_0)\cdot c_0)\bigl]\cdot c_0\biggl)\biggl)\cdot c_0\biggl)\biggl]\cdot c_0=c_0$.
        \end{enumerate}
    \end{enumerate}
    We call $\langle A;\cdot,c_0,(\nabla_i)_{i\in I}\rangle$ the \emph{diagonal-free $I$-dimensional cylindric implication algebra reduct} of $A$. 
\end{definition}
 
\begin{definition}
    Let $A$ be an $I$-dimensional cylindric implication algebra. Then define the following operation of substitution $\sigma^i_k\colon A\to A$ by $\sigma^i_k(x)=\nabla_i((d_{i,k}\cdot (x\cdot c_0))\cdot c_0)$.
\end{definition}

Proposition \ref{distributivity} and Proposition \ref{endomorphism for complements} will be exploited in the proof of Proposition \ref{endomorphism}. 
\begin{proposition}\label{distributivity}
    In any bounded implication algebra, we have the following: 
    \[\biggl[x\cdot\biggl(((y\cdot z)\cdot z)\cdot c_0\biggl)\biggl]\cdot c_0=\biggl[\biggl((x\cdot (y\cdot c_0))\cdot c_0\biggl)\cdot\biggl((x\cdot (z\cdot c_0))\cdot c_0\biggl)\biggl]\cdot(x\cdot (z\cdot c_0))\cdot c_0.\]
\end{proposition}
\begin{proof}
The result follows immediately by Theorem \ref{bia to ba}. Note that given our construction of meets and joins in the Boolean algebra induced by a bounded implication algebra, the above equation expresses the distributive identity $x\wedge(y\vee z)=(x\wedge y)\vee(x\wedge z)$ that is characteristic of Boolean algebras.  
\end{proof}
\begin{proposition}\label{endomorphism for complements}
Any $I$-dimensional cylindric implication algebra satisfies $\sigma^i_k(x\cdot c_0)=\sigma^i_k(x)\cdot c_0$.  
\end{proposition}
\begin{proof}
By Theorem \ref{bia to ba}, it suffices to show that $\sigma^i_k(x)\odot\sigma^i_k(x\cdot c_0)=c_0$ and $\sigma^i_k(x)\oplus\sigma^i_k(x\cdot c_0)=c_1$. The first equation is achieved by applying the definition of $\odot$ as well as condition 3(c) of Definition \ref{cia} as follows: 
\[\sigma^i_k(x)\odot\sigma^i_k(x\cdot c_0)=\biggl[\nabla_i\biggl((d_{i,k}\cdot(x\cdot c_0))\cdot c_0\biggl)\cdot\biggl(\biggl(\nabla_i\biggl(\bigl[d_{i,k}\cdot((x\cdot c_0)\cdot c_0)\bigl]\cdot c_0\biggl)\biggl)\cdot c_0\biggl)\biggl]\cdot c_0=c_0.\] For the second equation, Proposition \ref{preservation of joins} together with Proposition \ref{distributivity} yield:  
\begin{align*}
    \sigma^i_k(x)\oplus\sigma^i_k(x\cdot c_0)&=\biggl[\nabla_i\biggl((d_{i,k}\cdot(x\cdot c_0))\cdot c_0\biggl)\cdot\nabla_i\biggl(\bigl[d_{i,k}\cdot((x\cdot c_0)\cdot c_0)\bigl]\cdot c_0\biggl)  \biggl]\cdot\nabla_i\biggl(\bigl[d_{i,k}\cdot((x\cdot c_0)\cdot c_0)\bigl]\cdot c_0\biggl) 
    \\&=\nabla_i\biggl[\biggl(\biggl((d_{i,k}\cdot(x\cdot c_0))\cdot c_0\biggl)\cdot\biggl(\bigl[d_{i,k}\cdot((x\cdot c_0)\cdot c_0)\bigl]\cdot c_0\biggl)  \biggl)\cdot\biggl(\bigl[d_{i,k}\cdot((x\cdot c_0)\cdot c_0)\bigl]\cdot c_0\biggl)\biggl] 
    \\&=\nabla_i\biggl[\biggl(d_{i,k}\cdot\biggl(\bigl[(x\cdot (x\cdot c_0))\cdot(x\cdot c_0)\bigl]\cdot c_0\biggl)\biggl)\cdot c_0\biggl]
\end{align*}
Then by Theorem \ref{bia to ba}, Proposition \ref{constant1}, and the fact that $\nabla_id_{i,k}=c_1$, we obtain the following:  
\[\nabla_i\biggl[\biggl(d_{i,k}\cdot\biggl(\bigl[(x\cdot (x\cdot c_0))\cdot(x\cdot c_0)\bigl]\cdot c_0\biggl)\biggl)\cdot c_0\biggl]=\nabla_i\bigl[(d_{i,k}\cdot(c_1\cdot c_0))\cdot c_0\bigl]=\nabla_i((d_{i,k}\cdot c_0)\cdot c_0)=\nabla_id_{i,k}=c_1.\]
Hence we have $\sigma^i_k(x)\oplus\sigma^i_k(x\cdot c_0)=c_1$ and thus $\sigma^i_k(x\cdot c_0)=\sigma^i_k(x)\cdot c_0$, which completes the proof. 
\end{proof}
\begin{proposition}\label{endomorphism}
    For any $I$-dimensional cylindric implication algebra $A$, the operation $\sigma^i_k$ induces a bounded implication algebra endomorphism on $A$. 
\end{proposition}
\begin{proof}
It suffices to demonstrate that $\sigma^i_k(c_0)=c_0$ and $\sigma^i_k(x\cdot y)=\sigma^i_k(x)\cdot\sigma^i_k(y)$ for all $x,y\in A$. For the first equation, observe that Proposition \ref{constant1} together with Definition \ref{mbia}(1) gives:
    \[\sigma^i_k(c_0)=\nabla_i((d_{i,k}\cdot(c_0\cdot c_0))\cdot c_0)=\nabla_i((d_{i,k}\cdot c_1)\cdot c_0)=\nabla_i(c_1\cdot c_0)=\nabla_i(c_0)=c_0.\] For the second equation, note that the definition of $\sigma^i_k$ yields: 
\begin{equation}
    \sigma^i_k(x\cdot y)=\nabla_i((d_{i,k}\cdot((x\cdot y)\cdot c_0))\cdot c_0)
\end{equation}
 By Since $c_0\cdot y=c_1$, by Lemma \ref{lemma 1} we obtain the following: 
 \begin{equation}
     \nabla_i((d_{i,k}\cdot((x\cdot y)\cdot c_0))\cdot c_0)=\nabla_i((d_{i,k}\cdot(((((x\cdot c_0)\cdot y)\cdot y)\cdot c_0)))\cdot c_0)
 \end{equation}
    Hence by Equation (9), Equation (10), Proposition \ref{distributivity}, and Proposition \ref{preservation of joins}, we have: 
    \begin{align}
    \begin{split}
        \sigma^i_k(x\cdot y)&=\nabla_i\biggl[\biggl(d_{i,k}\cdot\biggl(\biggl(\bigl[((x\cdot c_0)\cdot y)\cdot y\bigl]\cdot c_0\biggl)\biggl)\biggl)\cdot c_0\biggl]
        \\&=\nabla_i\biggl[\biggl(\bigl[(d_{i,k}\cdot((x\cdot c_0)\cdot c_0))\cdot c_0\bigl]\cdot\bigl[(d_{i,k}\cdot(y\cdot c_0))\cdot c_0\bigl]\biggl)\cdot(d_{i,k}\cdot(y\cdot c_0))\cdot c_0\biggl]
        \\&=\biggl[\nabla_i\biggl(\bigl[d_{i,k}\cdot((x\cdot c_0)\cdot c_0)\bigl]\cdot c_0\biggl)\cdot\nabla_i\biggl((d_{i,k}\cdot(y\cdot c_0))\cdot c_0\biggl)\biggl]\cdot\nabla_i\bigl[(d_{i,k}\cdot(y\cdot c_0))\cdot c_0\bigl].
        \end{split}
    \end{align}
    Applying the definition of $\sigma^i_k$ to the right-hand side of Equation (11) then gives us: 
    \begin{equation}
        \biggl[\nabla_i\biggl(\bigl[d_{i,k}\cdot((x\cdot c_0)\cdot c_0)\bigl]\cdot c_0\biggl)\cdot\nabla_i\biggl((d_{i,k}\cdot(y\cdot c_0))\cdot c_0\biggl)\biggl]\cdot\nabla_i\bigl[(d_{i,k}\cdot(y\cdot c_0))\cdot c_0\bigl]=\bigl[\sigma^i_k(x\cdot c_0)\cdot\sigma^i_k(y)\bigl]\cdot\sigma^i_k(y).
    \end{equation}
Finally, Equations (11) and (12) together with Proposition \ref{endomorphism for complements} and Lemma \ref{lemma 1} yield:
    \[\sigma^i_k(x\cdot y)=\bigl[\sigma^i_k(x\cdot c_0)\cdot\sigma^i_k(y)\bigl]\cdot\sigma^i_k(y)=\bigl[(\sigma^i_k(x)\cdot c_0)\cdot \sigma^i_k(y)\bigl]\cdot\sigma^i_k(y)=\sigma^i_k(x)\cdot\sigma^i_k(y)\]
    Therefore, we conclude that $\sigma^i_k$ is a bounded implication algebra endomorphism on $A$.
\end{proof}
\begin{lemma}\label{lemma 3.20}
Any $I$-dimensional cylindric Boolean algebra satisfies: \[\exists_i\pi_c(\pi_c(\delta_{i,k},\pi_c(x,0)),0)=\exists_i(\delta_{i,k}\wedge x).\] 
\end{lemma}
\begin{proof}
By a repeated application of the definition of $\pi_c$ and the fact that $\pi_c(x,0)=x'$ yields the following:  \begin{equation}
\exists_i\pi_c(\pi_c(\delta_{i,k},\pi_c(x,0)),0)=\exists_i\pi_c(\pi_c(\delta_{i,k},x'),0)=\exists_i\pi_c(\delta_{i,k},x')'=\exists_i(\delta_{i,k}'\vee x')'
\end{equation}
Then by applying De Morgan's identities and the fact that $'$ is an involution yields: 
\begin{equation}
    \exists_i(\delta_{i,k}'\vee x')'=\exists_i(\delta_{i,k}''\wedge x'')=\exists_i(\delta_{i,k}\wedge x)
\end{equation}
 By Equations (13) and (14), we have $\exists_i\pi_c(\pi_c(\delta_{i,k},\pi_c(x,0)),0)=\exists_i(\delta_{i,k}\wedge x)$, as desired.  
\end{proof}
\begin{lemma}\label{lemma 3.21}
Any $I$-dimensional cylindric Boolean algebra satisfies: \[\exists_i\pi_c(\pi_c(\delta_{i,k},\pi_c(\pi_c(x,0),0)),0)=\exists_i(\delta_{i,k}\wedge x').\]  
\end{lemma}
\begin{proof}
    A similar calculation to that of Lemma \ref{lemma 3.20} gives the following:
\begin{align*}
    \exists_i\pi_c(\pi_c(\delta_{i,k},\pi_c(\pi_c(x,0),0)),0)&=\exists_i\pi_c(\pi_c(\delta_{i,k},x''),0)=\exists_i\pi_c(\delta_{i,k},x)'\\&=\exists_i(\delta'_{i,k}\vee x)'=\exists_i(\delta_{i,k}\wedge x')''\\&=\exists_i(\delta_{i,k}\wedge x')
\end{align*}
       which completes the proof. 
\end{proof}
\begin{theorem}\label{cba to cia}
    Every $I$-dimensional cylindric Boolean algebra can be converted into an $I$-dimensional cylindric implication algebra. 
\end{theorem}
\begin{proof}
    By Theorem \ref{mba to mia}, every monadic Boolean algebra can be converted into a monadic implication algebra and hence it suffices to verify that conditions 2 and 3(a)-3(c) of Definition \ref{cia} are satisfied. It is trivial that $\nabla_i\nabla_kx=\nabla_k\nabla_ix$ and $\delta_{i,k}=\delta_{k,i}$ as well as $\delta_{i,i}=1$ for all $i,k\in I$ by conditions 2 and 3(a) of Definition \ref{cia}. For condition 3(b), let $i,l\not=k$. It suffices to demonstrate $\exists_k\pi_c(\pi_c(\delta_{i,k},\pi_c(\delta_{k,l},0)),0)=\delta_{i,l}$. By setting $\delta_{i,k}=x$, Lemma \ref{lemma 3.20} along with Definition \ref{quantum cylindric algebra}(3(b)) yields $\exists_k\pi(\pi(\delta_{i,k},\pi_c(\delta_{k,l},0)),0)=\exists_i(\delta_{i,k}\wedge\delta_{k,l})=\delta_{i,l}$.

    We must now demonstrate the following: 
\begin{equation}
    \pi_c(\pi_c(\exists_i\pi_c(\pi_c(\delta_{i,k},\pi_c(x,0)),0),\pi_c(\exists_i\pi_c(\pi_c(\delta_{i,k},\pi_c(x,0),0),0),0)),0)=0
\end{equation}
to establish condition 3(c). Let $\phi$ and $\psi$ be defined in the following manner: \[\phi:=\exists_i\pi_c(\pi_c(\delta_{i,k},\pi_c(x,0)),0), \hspace{.2cm}\psi:=\exists_i\pi_c(\pi_c(\delta_{i,k},\pi_c(\pi_c(x,0),0)),0)\]  so that $\pi_c(\pi_c(\phi,\pi_c(\psi,0)),0)$ provides the left-hand side of Equation (15), i.e.,:  
\[\pi_c(\pi_c(\phi,\pi_c(\psi,0)),0)=\pi_c(\pi_c(\exists_i\pi_c(\pi_c(\delta_{i,k},\pi_c(x,0)),0),\pi_c(\exists_i\pi_c(\pi_c(\delta_{i,k},\pi_c(x,0),0),0),0)),0).\]
By applying Lemma \ref{lemma 3.20}, Lemma \ref{lemma 3.21}, as well as Definition \ref{quantum cylindric algebra}(3(c)), we obtain: 
\begin{align*}
    \pi_c(\pi_c(\phi,\pi_c(\psi,0)),0)&=\pi_c(\pi_c(\phi,y'),0)\\&=\pi_c((\phi'\vee\psi'),0)\\&=(\phi'\vee\psi')'=\phi\wedge\psi\\&=\exists_i(\delta_{i,k}\wedge x)\wedge\exists_i(\delta_{i,k}\wedge x')\\&=0
\end{align*}
Therefore Equation (15) is satisfied, which completes the proof. 
\end{proof}
\begin{theorem}\label{cia to cba}
    Every $I$-dimensional cylindric implication algebra can be converted into an $I$-dimensional cylindric Boolean algebra. 
\end{theorem}
\begin{proof}
    By Theorem \ref{theorem 3.13}, it suffices to show that every $I$-dimensional cylindric implication algebra satisfies conditions 2 and 3(a)-4(c) of Definition \ref{quantum cylindric algebra}. Condition 2 and 3(a) are immediate by conditions 2 and 3(a) of Definition \ref{cia}. For condition 3(b), assume that $i,k\not=l$. Then by Definition \ref{cia}(3(b)) we have: 
    \[\nabla_k(d_{i,k}\odot d_{k,l})=\nabla_{k}((d_{i,k}\cdot (d_{k,l}\cdot c_0))\cdot c_0)=d_{i,l},\] as desired. To see that Definition \ref{quantum cylindric algebra}(3(c)) is satisfied, assume that $i\not= k$ and note: 
    \[\nabla_i(d_{i,k}\odot x)\odot\nabla_{i}(d_{i,k}\odot -x)=\biggl[\nabla_i\biggl((d_{i,k}\cdot(x\cdot c_0))\cdot c_0\biggl)\cdot\biggl(\biggl(\nabla_i\biggl(\bigl[d_{i,k}\cdot((x\cdot c_0)\cdot c_0)\bigl]\cdot c_0\biggl)\biggl)\cdot c_0\biggl)\biggl]\cdot c_0=c_0\]
where the second equality follows by Definition \ref{cia}(3(c)). This completes the proof that every $I$-dimensional cylindric implication algebra can be converted into an $I$-dimensional cylindric Boolean algebra. 
\end{proof}
The result collected so far allow us to arrive at the following. 
\begin{theorem}
    Let $\mathfrak{B}=\langle B;\wedge,\vee,',0,1,(\exists_i)_{i\in I},(\delta_{i,k})_{i,k\in I}\rangle$ be an $I$-dimensional cylindric Boolean algebra, let $\mathfrak{B}^+=\langle B;\pi_c,0,(\exists_i)_{i\in I},(\delta_{i,k})_{i,k\in I}\rangle$ be the $I$-dimensional cylindric implication algebra obtained from $\mathfrak{B}$ via Theorem \ref{cba to cia}, and let $(\mathfrak{B}^{+})_+=\langle B;\odot,\oplus,-,0,1,(\exists_i)_{i\in I},(\delta_{i,k})_{i,k\in I}\rangle$ be the $I$-dimensional cylindric Boolean algebra obtained from $\mathfrak{B}^+$ via Theorem \ref{cia to cba}. Then $\mathfrak{B}=(\mathfrak{B}^+)_+$. 
\end{theorem}
\begin{theorem}
    Let $\mathfrak{A}=\langle A;\cdot,c_0,(\nabla_i)_{i\in I}, (d_{i,k})_{i,k\in I}\rangle$ be an $I$-dimensional cylindric implication algebra, let $\mathfrak{A}_+=\langle A;\odot,\oplus,-,c_0,c_1,(\nabla_i)_{i\in I},(d_{i,k})_{i,k\in I}\rangle$ be the $I$-dimensional cylindric Boolean algebra obtained from $\mathfrak{A}$ via Theorem \ref{cia to cba}, and let $(\mathfrak{A}_+)^+=\langle A;\pi_c,c_0,(\nabla_{i})_{i\in I},(d_{i,k})_{i,k\in I}\rangle$ be the I-dimensional cylindric implication algebra obtained from $\mathfrak{A}_+$ via Theorem \ref{cba to cia}. Then $\mathfrak{A}=(\mathfrak{A}_+)^+$. 
\end{theorem}
\begin{definition}
    Let $A$ and $B$ be $I$-dimensional cylindric implication algebras. Then a function $h\colon A\to B$ is a \emph{homomorphism} provided $h$ is a homomorphism from the monadic implication algebra reduct $\langle A;\cdot,c_0,\nabla_i\rangle$ of $A$ to the monadic implication algebra reduct $\langle B;\cdot,c_0,\nabla_i\rangle$ of $B$ and satisfies $h(d_{i,k})=d_{i,k}$.  
\end{definition}
\begin{definition}
    By $\mathbf{CIA}$ we denote the category of $I$-dimensional cylindric implication algebras and homomorphisms and by $\mathbf{CBA}$ we denote the category of $I$-dimensional cylindric Boolean algebras and homomorphisms. 
 \end{definition}
\begin{theorem}\label{theorem 3.28}
    $\mathbf{CIA}$ is isomorphic to $\mathbf{CBA}$. 
\end{theorem}
\section{Spectral Duality for $I$-dimensional cylindric implication algebras}
As an immediate application of the isomorphism established in the previous section between $\mathbf{CIA}$ and $\mathbf{CBA}$, we proceed by establishing a dual equivalence between $\mathbf{CIA}$  and the category $\mathbf{CUV}$ of cylindric upper Vietoris spaces and spectral $p$-morphisms, which were introduced in \cite{mcdonald2} as the choice-free topological duals of cylindric Boolean algebras, and are relational extensions of the upper Vietoris spaces introduced in \cite{bez}, the choice-free topological duals of the Boolean algebras.
\subsection{Cylindric Upper Vietoris Spaces}
 A fundamental concept underlying these dualities is that of a possibility frame (see \cite{bez, holliday}). A \emph{possibility frame} is a triple $\langle X;\leq,P\rangle$ such that $\langle X;\leq\rangle$ is a poset and $XP$ is the regular open subsets in the upset topology on $\langle X;\leq\rangle$ where $X\in P$ and $P$ is closed under $\cap$ and $^*$ where: 
\[U^*=\{x\in X:x'\not\in U\hspace{.2cm}\text{for all $x'\geq x$}\}.\]
Recall that if $X$ is a topological space, then an open set $U\subseteq X$ is \emph{regular open} iff $\text{Int}(\text{Cl}(U))=U$ where $\text{Int}$ and $\text{Cl}$ are the operations of interior and closure on $X$. In the upset topology $\text{UP}(X;\leq)$, we have: 
\[\text{Int}_{\leq}(U)=\{x\in X:y\in U\hspace{.2cm}\text{for all $y\geq x$}\},\hspace{.3cm}\text{Cl}_{\leq}(U)=\{x\in X:y\in U\hspace{.2cm}\text{for some $y\geq x$}\}\]
 so $\text{Int}_{\leq}(U)=X\setminus\text{Cl}_{\leq}(X\setminus U)$. Hence, an open set $U$ is \emph{regular open} in the upset topology on $X$ iff $\text{Int}_{\leq}(\text{Cl}_{\leq}(U))=U$. Since $U^*=\text{Int}_{\leq}(X\setminus U)$, an open set $U$ is \emph{regular open} in $\text{UP}(X;\leq)$ iff $U^{**}=U$. By $\mathcal{REG}(X)$, we denote the collection of all such regular opens, by $\mathcal{C}(X)$ we denote the collection of compact subsets of $X$, and by $\mathcal{O}(X)$ we denote the collection of open subsets of $X$ so that: 
\[\mathcal{CO}(X)=\mathcal{C}(X)\cap\mathcal{O}(X),\hspace{.2cm}\mathcal{COREG}(X)=\mathcal{CO}(X)\cap\mathcal{REG}(X).\]
The following class of spaces was introduced by Bezhanishvili and Holliday \cite{bez} and form the choice-free topological duals of the Boolean algebras. Let: 
\[\mathcal{COREG}_X(x)=\{U\in\mathcal{COREG}(X):x\in U\}.\]

\begin{definition}[\cite{bez}]\label{uv space}
    A \emph{UV-space} is a $T_0$-space $X$ satisfying the following conditions: 
    \begin{enumerate}
        \item $\mathcal{COREG}(X)$ is closed under $\cap$, $\text{Int}_{\leqslant}(X\setminus\cdot)$, and is a basis for $X$; 
        \item every proper filter in $\mathcal{COREG}(X)$ is $\mathcal{COREG}_X(x)$ for some $x\in X$. 
    \end{enumerate}
\end{definition}
Recall that a topological space $X$ is \emph{coherent} if $\mathcal{CO}(X)$ is closed under $\cap$ and forms a basis for $X$ and \emph{sober} if every completely prime filter in $\mathcal{CO}(X)$ is of the form $\mathcal{CO}_X(x)$ for some $x\in X$ where: \[\mathcal{CO}_X(x):=\{U\in\mathcal{CO}(X):x\in U\}.\] A \emph{spectral space} is a compact $T_0$-space that is coherent and sober. It follows by \cite[Corollary 5.5]{bez} that every UV-space is a spectral space and hence the specialization order $\leqslant$ associated with a UV-space $X$ defined by $x\leqslant y$ iff $x\in U$ implies $y\in U$ for all $U\in\mathcal{O}(X)$ is a partial order, since every spectral space is a $T_0$-space.

The following were introduced by McDonald \cite{mcdonald2} are form relational extensions of the UV-spaces described above. For the purposes of the following definition, let:  
\[\exists_{S_i}[\mathcal{COREG}_X(x)]:=\{\exists_{S_i}U:U\in\mathcal{COREG}_X(x)\}.\]

\begin{definition}
     A \emph{monadic UV-space} is a relational topological space of the following shape $\langle X;S,\tau\rangle$ such that the following conditions are satisfied: 
     \begin{enumerate}
         \item $\langle X;\tau\rangle$ is a UV-space; 
         \item $S$ is an equivalence relation; 
         \item $S[U]\in\mathcal{COREG}(X)$ whenever $U\in\mathcal{COREG}(X)$;
        \item if $x\overline{S}y$, then $\exists_{S}[\mathcal{COREG}_X(x)]\not=\exists_{S}[\mathcal{COREG}_X(y)]$.  
     \end{enumerate}
\end{definition}

\begin{definition}\label{cylindric uv space}
     An \emph{I-dimensional cylindric UV-space} is a space $\langle X;(S_i)_{i\in I},(\Delta_{i,k})_{i,k\in I},\tau\rangle$ such that: 
    \begin{enumerate}
        \item $\langle X;S_i,\tau\rangle$ is a monadic UV-space for each $i\in I$; 
      \item $S_i$ commutes with $S_k$, i.e., $S_i\circ S_k=S_k\circ S_i$ for all $i,k\in I$; 
        \item $\Delta_{i,k}\subseteq X$ with $\Delta^{**}_{i,k}=\Delta_{i,k}=\Delta_{k,i}$ and $\Delta_{i,i}=X$;
        \item if $i,l\not=k$, then $S_k[\Delta_{i,k}\cap \Delta_{k,l}]=\Delta_{i,l}$;
        \item if $i\not=k$, then $S_i[\Delta_{i,k}\cap U]\cap S_i[\Delta_{i,k}\cap U^*]=\emptyset$;   
        \item $\Delta_{i,k}\in\mathcal{CO}(X)$. 
    \end{enumerate}
\end{definition}

 \subsection{Duality Theory} In this final section, we obtain spectral duality results for $I$-dimensional cylindric implication algebras under the established isomorphism between $\mathbf{CIA}$ and $\mathbf{CBA}$ together with the duality results established in \cite{mcdonald2} for $\mathbf{CBA}$.  
\begin{definition}\label{filters}
    Let $A$ be a bounded implication algebra. A non-empty subset $\alpha\subseteq A$ is a filter provided: 
    \begin{enumerate}
        \item if $x\in\alpha$ and $x\cdot y=1$, then $y\in\alpha$; 
        \item if $x,y\in\alpha$, then $(x\cdot(y\cdot 0))\cdot 0\in\alpha$. 
    \end{enumerate}
    Moreover, if $\alpha$ is a filter of $A$, then $\alpha$ is a \emph{proper filter} if $\alpha\not=A$. 
\end{definition}
\begin{remark}
    It is obvious by Definition \ref{definition 2.7}, Theorem \ref{bia to poset}, and Theorem \ref{theorem 2.15} that the above definition of a (proper) filter of a bounded implication algebra corresponds to that of a (proper) filter in a Boolean algebra. 
\end{remark}

\begin{definition}\label{spectrum}
    If $A$ is an $I$-dimensional cylindric implication algebra define the \emph{spectrum} of $A$ as the topological space $X_A=\langle\mathfrak{F}(A);(S_i)_{i\in I},(\Delta_{i,k})_{i,k\in I},\tau\rangle$ such that: 
    \begin{enumerate}
        \item $\mathfrak{F}(A)$ is the collection of all proper filters of $A$; 
        \item $\tau$ is the topology generated by the basis $\mathcal{B}=\bigcup_{x\in A}\phi(x)$ where: \[\phi(x)=\{\alpha\in\mathfrak{F}(A):x\in \alpha\};\]  
        \item $S_i\subseteq\mathfrak{F}(A)\times\mathfrak{F}(X)$ is defined by: \[\alpha S_i\beta\Longleftrightarrow\{\nabla_ix:x\in\alpha\}=\{\nabla_ix:x\in\beta\};\] 
        \item $\Delta_{i,k}\subseteq\mathfrak{F}(A)$ is defined by $\Delta_{i,k}=\phi(d_{i,k})$. 
    \end{enumerate}
\end{definition}
For the purposes of Theorem \ref{rep theorem}, recall that if $X$ and $X'$ are topological spaces that come equipped with families of binary relations $(S_i)_{i\in I}\subseteq X\times X$ and $(S'_i)_{i\in I}\subseteq X'\times X'$, then a function $f\colon X\to X'$ is a \emph{relational homeomorphism} if $f$ is a homeomorphism satisfying $xS_iy$ iff  $f(x)S'_if(y)$ for all $i\in I$. 
\begin{theorem}\label{rep theorem}
    Let $A$ be an $I$-dimensional cylindric implication algebra and let $X$ be an $I$-dimensional cylindric UV-space. Then:
    \begin{enumerate}
    \item the spectrum $X_A$ is an $I$-dimensional cylindric UV-space; 
    \item the algebra $\mathcal{COREG}(X)$ is an $I$-dimensional cylindric implication algebra
        \item $A$ is isomorphic to $\mathcal{COREG}(X_A)$; 
        \item $X$ is relationally homeomorphic to $X_{\mathcal{COREG}(X)}$. 
    \end{enumerate}
\end{theorem}
\begin{proof}
Part 1 follows immediately by Theorem \ref{cia to cba} together with \cite[Lemma 7.5(1)]{mcdonald2}. For part 2, it follows by \cite[Lemma 7.5(2)]{mcdonald2} that the algebra $\langle\mathcal{COREG}(X);\cap,\sqcup,^*,\emptyset,X,(\exists_{S_i})_{i\in I},(\Delta_{i,k})_{i,k\in I}\rangle$ is an $I$-dimensional cylindric Boolean algebra under $U\sqcup V:=\text{Int}_{\leqslant}(\text{Cl}_{\leqslant}(U\cup V))$ and $U^*:=\text{Int}_{\leqslant}(X\setminus U)$, and hence by Theorem \ref{cba to cia} that $\langle\mathcal{COREG}(X);\pi_c,\emptyset,(\exists_{S_i})_{i\in I},(\Delta_{i,k})_{i,k\in I}\rangle$ is an $I$-dimensional cylindric implication algebra under $\pi_c(U,V):=U^*\sqcup V$. Then Theorem \ref{theorem 3.28} combined with \cite[Theorem 7.7]{mcdonald2} imply that every $I$-dimensional cylindric implication algebra $A$ is isomorphic to $\mathcal{COREG}(X_A)$ under $\phi(x)=\{\alpha\in\mathfrak{F}(A):x\in\alpha\}$ where $\phi(x\cdot y)=\pi_c(\phi(x),\phi(y))$, $\phi(c_0)=\emptyset$, $\phi(\nabla_ix)=\nabla_i\phi(x)$, $\phi(d_{i,k})=\Delta_{i,k}$, 
and that every $I$-dimensional cylindric UV-space $X$ is relationally homeomorphic to $X_{\mathcal{COREG}(X)}$ under $\psi(x)=\{U\in\mathcal{COREG}(X):x\in U\}$.   
\end{proof}
\begin{definition}\label{uv map}
    Let $X$ and $Y$ be $I$-dimensional cylindric UV-spaces. A function $f\colon X\to X'$ is a \emph{cylindric UV-map} if the following conditions are satisfied: 
    \begin{enumerate}
        \item $f$ is a spectral map, i.e., $f^{-1}[U]\in\mathcal{CO}(X)$ for each $U\in\mathcal{CO}(Y)$ 
        \item if $f(x)\leqslant' y'$, there exists $y\in X$ such that $x\leqslant y$ and $f(y)=y'$; 
        \item $S_{i}[f^{-1}[U]]=f^{-1}[S_{i}[U]]$ for all $U\in\mathcal{COREG}(X)$; 
   \item $f^{-1}[\Delta'_{i,k}]=\Delta_{i,k}$. 
    \end{enumerate}
\end{definition}

If $X$ and $Y$ are UV-spaces and $f\colon X\to X'$ is a function, then $f$ is a \emph{UV-map} provided conditions 1 and 2 in Definition \ref{uv map} are satisfied. Moreover, if $X$ and $Y$  are monadic UV-spaces, then $f\colon X\to Y$ is a \emph{monadic UV-map} provided conditions 1-3 are satisfied.
\begin{definition}
   By $\mathbf{UV}$ we denote the category of UV-spaces and UV-maps, by $\mathbf{MUV}$ we denote the category of monadic UV-spaces and monadic UV-maps, and by $\mathbf{CUV}$ we denote the category of cylindric UV-spaces and cylindric UV-maps.  
\end{definition}

 \begin{proposition}[\protect{\cite[Corollary 6.4]{bez}}]\label{preimage of reg is reg}
    Let $X$ and $X'$ be UV-spaces and let $f\colon X\to X'$ be a UV-map. Then $f^{-1}[U]\in\mathcal{COREG}(X)$ for each $U\in\mathcal{COREG}(X')$. 
\end{proposition}
\begin{theorem}\label{matching maps}
    Let $A$ and $B$ be $I$-dimensional cylindric implication algebras and let $h\colon A\to B$ be a homomorphism. Then there is an $I$-dimensional cylindric UV-map $f\colon X_B\to X_A$ where $X_B$ and $X_A$ are the $I$-dimensional cylindric UV-spaces induced by $B$ and $A$. Conversely, if $X$ and $Y$ are $I$-dimensional cylindric UV-spaces and $f\colon X\to Y$ is an $I$-dimensional cylindric UV-map, then there is a homomorphism $h\colon\mathcal{COREG}(Y)\to\mathcal{COREG}(X)$ where $\mathcal{COREG}(Y)$ and $\mathcal{COREG}(X)$ are the $I$-dimensional cylindric implication algebras induced by $Y$ and $X$. 
\end{theorem}
\begin{proof}
    For part 1, set $f=h^{-1}$ and apply Theorem \ref{theorem 3.28} together with \cite[Theorem 7.11]{mcdonald2} and analogously for part 2, set $h=f^{-1}$ and apply Theorem \ref{theorem 3.28} together with \cite[Theorem 7.11]{mcdonald2}. 
\end{proof}
\begin{theorem}
    $\mathbf{CIA}$ is dually equivalent to $\mathbf{CUV}$. 
\end{theorem}

\begin{remark}
    We conclude by noting that under Theorem \ref{theorem 2.15} and Theorem \ref{theorem 3.18} together with obvious restrictions to Definition \ref{spectrum}, Theorem \ref{rep theorem}, and Theorem \ref{matching maps} provides a dual equivalence between $\mathbf{IA_{0,1}}$ and the category $\mathbf{UV}$ of UV-spaces and UV-maps as well as between $\mathbf{MIA}$ and the category $\mathbf{MUV}$ of monadic UV-spaces and monadic UV-maps. 
\end{remark}
\section{Conclusions}
We have extended the work of Abbott \cite{abbott672} by introducing monadic and cylindric expansions of bounded implication algebras and demonstrating that $\mathbf{IA_{0,1}}$ is isomorphic to $\mathbf{BA}$, $\mathbf{MIA}$ is isomorphic to $\mathbf{MBA}$, and $\mathbf{CIA}$ is isomorphic to $\mathbf{CBA}$. As an immediate application, we have shown that $\mathbf{CIA}$ is dually equivalent to $\mathbf{CUV}$. An obvious restriction of this duality also yields a dual equivalence between $\mathbf{IA_{0,1}}$ and $\mathbf{UV}$ as well as between $\mathbf{MIA}$ and $\mathbf{MUV}$.   
\section*{Declarations}
\paragraph{\textbf{Funding}}
This work has been funded by a grant from the Programme Johannes Amos Comenius
under the Ministry of Education, Youth and Sports of the Czech Republic,
CZ.02.01.01/00/23$_{-}$025/0008711. 
\paragraph{\textbf{Ethical approval}}
Not applicable. 

\paragraph{\textbf{Availability of data and materials}}
Not applicable. 

\paragraph{\textbf{Competing interests}} The author declares no competing interests.


\begin{thebibliography}{99}
 \setlength{\itemsep}{0em}
    \setlength{\parskip}{0em}



\bibitem{abad} Abad, M., Cimadamore, C., Díaz Varela, J.: Topological representation for monadic implication algebras. Central European Journal of Mathematics. \textbf{7}, 299--309 (2009). \url{https://doi.org/10.2478/s11533-009-0002-y}
    
    \bibitem{abbott671} Abbott, J. C.: Semi-Boolean Algebra. \emph{Matematicki Vesnik}. \textbf{4}, 177-198 (1967) 
    \bibitem{abbott672} Abbott, J. C.: Implicational algebras. \emph{Bulletin math\'ematique de la Soci\'et\'e des Sciences Mathématiques de la R\'epublique Socialiste de Roumanie}. \textbf{11}, 3--23 (1697)
   



\bibitem{bez} Bezhanishvili, N., Holliday, W.: Choice-free Stone duality. \emph{Journal of Symbolic Logic}, \textbf{85}, 109--148 (2020). \url{https://doi.org/10.1017/jsl.2019.11} 






 \bibitem{halmos} Halmos, P.: Algebraic Logic. Chelsea Publishing Company, New York (1962)

 \bibitem{tarski1} Henkin, L., Monk, J.D., Tarski, A.: Cylindric algebras. Part I, volume 64 of Studies
in Logic and the Foundations of Mathematics. North-Holland Publishing Co., Amsterdam (1985)

\bibitem{tarski2} Henkin, L., Monk J.D., and Tarski, A.: Cylindric algebras. Part II, volume 115 of Studies
in Logic and the Foundations of Mathematics. North-Holland Publishing Co., Amsterdam (1985)


\bibitem{holliday} Holliday, W.: Possibility frames and forcing for modal logic. The Australasian journal of Logic. \textbf{22}, 44 -- 288 (2025) 

 \bibitem{mcdonald2} McDonald, J.: Canonical completion and duality for cylindric ortholattices and cylindric Boolean algebras. \emph{Studia Logica}, 2026. \url{ https://doi.org/10.1007/s11225-026-10234-z}


    \end{thebibliography}
\end{document}